\documentclass{amsart}

\usepackage{times}
\usepackage{amsmath,amssymb,amsthm}
\usepackage{mathrsfs}
\usepackage{xcolor}
\usepackage{tikz-cd}
\usepackage[margin=1in]{geometry}
\usepackage[dvipsnames]{xcolor}
\usepackage[pdftex,%
    final,%
    colorlinks=true,
    linkcolor=NavyBlue,
    citecolor=NavyBlue,
    filecolor=NavyBlue,
    menucolor=NavyBlue,
    urlcolor=NavyBlue,
    bookmarks=true,
    bookmarksdepth=3,
    bookmarksnumbered=true,
    bookmarksopen=true,
    bookmarksopenlevel=2,%
]{hyperref}

\newcommand{\Q}{\mathbb{Q}}
\newcommand{\HFK}{\mathit{HFK}}
\newcommand{\CFK}{\mathit{CFK}}
\newcommand{\Tor}{\operatorname{Tor}}
\newcommand{\rk}{\operatorname{rk}}
\newcommand{\QUn}{\frac{\Q[U]}{(U^n)}}
\newcommand{\tensQU}{\underset{\Q[U]}{\otimes}}

\theoremstyle{plain}
\newtheorem{theorem}{Theorem}[section]
\newtheorem{lemma}[theorem]{Lemma}
\newtheorem{corollary}[theorem]{Corollary}
\newtheorem{proposition}[theorem]{Proposition}
\theoremstyle{definition}
\newtheorem{definition}[theorem]{Definition}
\newtheorem{conjecture}[theorem]{Conjecture}
\newtheorem{remark}[theorem]{Remark}
\newtheorem{example}[theorem]{Example}

\title{$L$-space knots and $\mathfrak{sl}_n$-like knot Floer homology}
\author{Tristan Wells}
\address{Department of Mathematics, Louisiana State University, LA 70803}
\email{trwells22@lsu.edu}

\begin{document}

\begin{abstract}
It has been of particular interest in recent years to relate categorifications of quantum knot invariants like the Jones polynomial to Floer-theoretic invariants like knot Floer homology using spectral sequences. In pursuit of Rasmussen's conjecture, Dowlin constructed a family of invariants with properties like those of Khovanov and Rozansky's $\mathfrak{sl}_n$-homology theories. In this paper, we explore these invariants for $L$-space knots and thin knots. In particular, for $n\geq 2$ the invariants of thin knots can be computed from a change of coefficients from the minus flavor of knot Floer homology, and the same holds for an $L$-space knot exactly when an explicit condition relating $n$ to its Alexander polynomial is satisfied.
\end{abstract}

\maketitle

\section{Introduction}

The relationship between Floer-theoretic knot invariants of Ozsv\'ath and Szab\'o \cite{OS04} and of Rasmussen \cite{Ras03} and quantum knot invariants of Khovanov \cite{Kho00} and of Khovanov and Rozansky \cite{KR08a, KR08b} has been an active topic in topology for the past twenty-five years. These invariants categorify many knot invariants. That is, the graded Euler characteristic of the knot Floer homology of a knot returns its Alexander polynomial, and the graded Euler characteristic of the Khovanov homology of a knot returns its Jones polynomial. The first relationship between them came from a spectral sequence from the Khovanov homology of a link to the Heegaard Floer homology of its double branched cover, constructed by Ozsv\'ath and Szab\'o in \cite{OS05b}. Since then, a slew of spectral sequences of that form have been discovered, yielding a variety of rank inequalities between the associated homologies as well as indicating the possibility of more interesting relationships.

Consequently, the following was a long-standing conjecture of Rasmussen \cite{Ras05}.

\begin{theorem}[Rasmussen's Conjecture]\label{thm:rasmussen}
\cite{Dow24} For a knot $K\subset S^3$, $\rk(\widetilde{Kh}(K;\Q))\geqslant \rk(\widehat{\HFK}(K;\Q))$, where $\widetilde{Kh}$ denotes reduced Khovanov homology.
\end{theorem}

\begin{figure}[ht]
    \centering
    \begin{tikzcd}[column sep=large, row sep=tiny]
        H(L)\arrow[r]\arrow[dd, bend right=30]\arrow[dddd, bend right=45]\arrow[ddddd, bend right=55] & HFK(L)\arrow[dd, bend left=30]\arrow[dddd, bend left=45]\arrow[ddddd, bend left=55]\\
        \vdots & \vdots\\
        KR_n(L) \arrow[r, dotted]& \HFK_n(L)\\
        \vdots & \vdots\\
        KR_{3}(L)\arrow[r, dotted] & HFK_{3}(L)\\
        KR_{2}(L)=Kh(L) \arrow[r] & HFK_{2}(L)\\
    \end{tikzcd}
    \caption{Dotted arrows are conjectured, while solid arrows have been proven in suitable versions: the curved arrows on the left are Rasmussen's spectral sequences \cite{Ras15}, those on the right are due to Dowlin \cite[Lemma 2.23]{Dow18}, and the bottom arrow is Dowlin's spectral sequence from reduced Khovanov homology to $\widehat{\HFK}$ \cite{Dow24}.}
    \label{fig:diagram}
\end{figure}
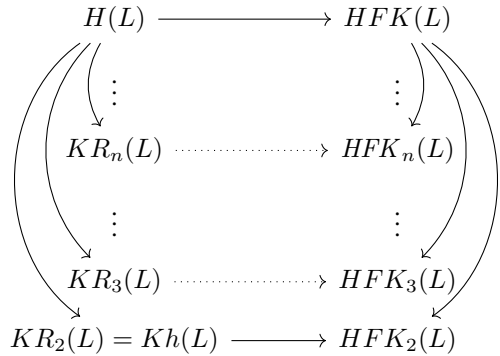

In \cite{Ras05}, Rasmussen suggests that constructing a spectral sequence is a promising way to address this conjecture, a strategy successfully employed by Dowlin in \cite{Dow18, Dow24}. It is natural to want to complete the diagram in Figure~\ref{fig:diagram} with spectral sequences from the Khovanov-Rozansky homologies, $KR_n(K)$, to a proposed Floer-theoretic analog. In \cite{Dow18}, Dowlin defines a candidate family of link invariants, called $\HFK_n(K)$, and presents the generalized conjecture.

\begin{conjecture}\label{conj:dowlin}
\cite[Conjecture 1.6]{Dow18} For a link $L\subset S^3$, there are spectral sequences $KR_n(L)\rightrightarrows \HFK_n(L)$ and $\overline{KR}_n(L)\rightrightarrows \widehat{\HFK}_n(L)$, respecting the gradings specified in \cite[Conjecture 1.6]{Dow18}.
\end{conjecture}

Motivated by Conjecture~\ref{conj:dowlin}, we study $\HFK_n(L)$, in the case that $L$ is a knot, and make the following conjecture.

\begin{conjecture}\label{conj:main}
For any knot $K$ in $S^3$, there is an integer $N(K)$ such that for all $n\geq N(K)$, $\HFK_n(K)$ can be obtained by a change of coefficients from $\HFK^-(K)$; that is,
\[
\HFK_n(K)\cong \big(\HFK^-(K)\tensQU \QUn\big)\oplus \Tor_1^{\Q[U]}\Big(\HFK^-(K); \QUn\Big).
\]
\end{conjecture}

Although further investigation is required to prove Conjecture~\ref{conj:main}, much can be said about two wide classes of knots. Section~\ref{sec:examples} presents supporting examples, including the trefoil (an $L$-space knot) and the figure-eight knot (a thin knot), as well as $T_{3,4}$. We recall that a rational homology 3-sphere $Y$ is an \emph{$L$-space} if $\rk\widehat{HF}(Y;\Q)=|H_1(Y;\mathbb{Z})|$.\footnote{Since $\rk\widehat{HF}(Y;\mathbb{F}_2)\geqslant\rk\widehat{HF}(Y;\Q)\geqslant|H_1(Y;\mathbb{Z})|$ by the universal coefficient theorem and the Euler characteristic, an $L$-space over $\mathbb{F}_2$ is also an $L$-space over $\Q$.} A knot is an \emph{$L$-space knot} if it admits a positive surgery yielding an $L$-space. By \cite{OS05a}, the Alexander polynomial of an $L$-space knot has the form $\Delta_K(t)=\sum_{j=0}^{m-1}(-1)^jt^{e_j}$ with $e_0>e_1>\cdots>e_{m-1}$; we call the differences $e_{j-1}-e_j$ the \emph{gaps} of $\Delta_K$. A knot $K$ is \emph{Floer homologically thin} if its knot Floer homology $\widehat{\HFK}(K;\Q)$ is supported in a single $\delta$-grading. Indeed, our main theorem offers some initial evidence in the direction of Conjecture~\ref{conj:main}:

\begin{theorem}\label{thm:main-intro}
Let $n\geq 2$, and let $K\subset S^3$ be either a thin knot, or an $L$-space knot such that every gap of $\Delta_K$ is at most $n-1$. Then, as ungraded $\Q[U]$-modules,
\[
\HFK_n(K)\cong \big(\HFK^-(K)\tensQU \QUn\big)\oplus \Tor_1^{\Q[U]}\Big(\HFK^-(K); \QUn\Big).
\]
\end{theorem}

That is to say, for such knots, $\HFK_n(K)$ is a change of coefficients from $\HFK^-(K)$.

\begin{corollary}\label{cor:determined}
For any $L$-space knot $K$, $\HFK_n(K)$ is determined by its Alexander polynomial. For any thin knot $K$, $\HFK_n(K)$ is determined by its Alexander polynomial and $\tau(K)$; moreover, for $n\geq 2$,
\[
\HFK_n(K)\cong \QUn\oplus\Q^{\det(K)-1}
\]
as ungraded $\Q[U]$-modules.
\end{corollary}

While the first part of Corollary~\ref{cor:determined} is implicitly known from the definition of $\HFK_n(L)$ for knots and does not depend on Theorem~\ref{thm:main-intro}, this follows more explicitly from \cite{OS05a}, where we see that the knot Floer complex (and hence $\HFK^-(K)$) of any $L$-space knot is determined by its Alexander polynomial, and from \cite{Pet13}, where we see that the knot Floer complex of any thin knot is determined from its Alexander polynomial and $\tau$ (see Theorems~\ref{thm:thin-structure} and~\ref{thm:lspace-structure}). For quasi-alternating knots $\tau=-\sigma/2$ \cite{MO08}, so for these knots the signature may be used in place of $\tau$. The second part follows from Theorem~\ref{thm:main-intro} (see Remark~\ref{rem:thin-formula}). The work in \cite{Pet13} will play a more critical role in the proof for thin knots in Section~\ref{sec:proof}.

The restriction to large $n$ in Theorem~\ref{thm:main-intro} is necessary. For $n=1$, $\HFK_1(K)\cong\widehat{HF}(S^3)\cong\Q$ for every knot (Remark~\ref{rem:n=1}), while the right-hand side is $\widehat{\HFK}(K)$. For $L$-space knots, Proposition~\ref{prop:sharp} computes $\HFK_n(K)$ for every $n\geq2$, and shows that the isomorphism of Theorem~\ref{thm:main-intro} holds if and only if an explicit count read off from the gaps of $\Delta_K$, vanishes. For torus knots this gives a clean answer.

\begin{corollary}\label{cor:torus}
Let $2\leq p<q$ be coprime and $n\geq2$. The isomorphism of Theorem~\ref{thm:main-intro} holds for $T_{p,q}$ if and only if $n\geq p$. In particular, it fails for $T_{3,4}$ with $n=2$, where $\dim_\Q\HFK_2(T_{3,4})=6$ while the right-hand side has dimension $8$.
\end{corollary}

\subsection*{Organization}
In Section~\ref{sec:background}, we introduce and review a construction of the invariants $\HFK_n(K)$ as in \cite{Dow18}. Section~\ref{sec:examples} compares several computations of these invariants against older knot Floer homology flavors for representatives from various classes of knots. In Section~\ref{sec:proof}, we prove the main theorem.

\subsection*{Tools Used}
All results of this note are due to the author. Claude Opus 5.5 was used in proofreading, checking proof details, and generating and checking examples.

\subsection*{Acknowledgments}
This note is modified from the author's Ph.D. dissertation at Michigan State University.

I am eternally grateful to my Ph.D. advisor Matthew Hedden for helpful conversation and comments on early drafts. I thank Charles Stine and Shea Vela-Vick for helpful conversations. I also thank Chris St. Clair, Dean Spyropoulos, and Chen Zhang for comments on early drafts. 

I also want to thank the National Science Foundation for support with grants
DMS-2104664, the RTG NSF Award DMS-2135960, DMS-1709016, and the RTG NSF Award DMS-2231492. Their financial support made this research possible.

\section{Background}\label{sec:background}

\subsection{Heegaard Diagrams}
A common way to depict 3-manifolds in low-dimensional topology is using a Heegaard diagram, which presents a convenient Morse decomposition of the 3-manifold. These can be enhanced to encode embedded links in the 3-manifold as well.

\begin{definition}\label{def:HD}
(cf.\ \cite[Definition 2.1]{Dow18}) A (multi-pointed) \emph{Heegaard diagram} $\mathcal{H}$ for an oriented link $L$ in a closed, oriented 3-manifold $Y$ is the data $(\Sigma,\boldsymbol{\alpha},\boldsymbol{\beta},\mathbf{w},\mathbf{z})$ where
\begin{itemize}
\item $\Sigma$ is a closed, oriented, genus $g$ surface, called the Heegaard surface.
\item $\boldsymbol{\alpha}$ (respectively, $\boldsymbol{\beta}$) is a set of $g+k-1$ disjoint embedded circles $\{\alpha_1,\dots,\alpha_{g+k-1}\}$ (resp.\ $\{\beta_1,\dots,\beta_{g+k-1}\}$) in $\Sigma$ such that the alpha curves intersect the beta curves transversely.
\item $\boldsymbol{\alpha}$ and $\boldsymbol{\beta}$ each span a $g$-dimensional subspace of $H_1(\Sigma;\mathbb{Z})$.
\item $\mathbf{w}$ and $\mathbf{z}$ are each sets of $k$ basepoints $\{w_1,\dots,w_k\}$ and $\{z_1,\dots,z_k\}$, where $k$ is at least the number of components of $L$, such that each component of $\Sigma\setminus\boldsymbol{\alpha}$ and each component of $\Sigma\setminus\boldsymbol{\beta}$ contains exactly one $\mathbf{w}$ basepoint and exactly one $\mathbf{z}$ basepoint.
\end{itemize}
\end{definition}

The manifold $Y$ is constructed from this data by thickening $\Sigma$ to $\Sigma\times[0,1]$ and attaching thickened discs along the alpha curves at $\Sigma\times\{0\}$ and along the beta curves at $\Sigma\times\{1\}$ and then filling in the resulting 2-sphere boundaries with 3-balls.

The link $L$ in $Y$ can be seen as follows. When thickening $\Sigma$, also thicken the $\mathbf{w}$ and $\mathbf{z}$ basepoints to get $\mathbf{w}\sqcup\mathbf{z}\times[0,1]\subset\Sigma\times[0,1]$. Since each component of $\Sigma\setminus\boldsymbol{\alpha}$ has a $\mathbf{w}$ and $\mathbf{z}$ basepoint, we connect those arcs on the $\Sigma\times\{0\}$ surface with an arc in the 3-ball attached to that component's $S^2$ boundary. We do this similarly for the $\Sigma\times\{1\}$ boundaries.

\begin{definition}\label{def:punctured}
In \cite[Definition 2.2]{Dow18}, Dowlin defined a notion of a \emph{punctured Heegaard diagram}, which adds to $(\Sigma,\boldsymbol{\alpha},\boldsymbol{\beta},\mathbf{w},\mathbf{z})$ the following data:
\begin{itemize}
\item A basepoint $p$, called the \emph{puncture}.
\item A curve $\alpha_{g+k}$ and $\beta_{g+k}$ which separate $p$ from the $\mathbf{w}$ and $\mathbf{z}$ basepoints, such that $\boldsymbol{\alpha}$ and $\boldsymbol{\beta}$ still intersect transversely and span a $g$-dimensional subspace of $H_1(\Sigma;\mathbb{Z})$.
\end{itemize}
\end{definition}

\begin{remark}\label{rem:puncture}
In the above construction of the pair $(Y,L)$, puncturing the Heegaard diagram is like considering $L\sqcup\mathcal{U}$ (where $\mathcal{U}$ is the unknot), where the unlinked unknot carries a pair of basepoints $w_0$ and $z_0$ which coincide at $p$; see \cite[Remark 2.3]{Dow18}.
\end{remark}

\subsection{Knot Floer Homology}
Here, we provide the groundwork for Dowlin's new family of invariants in \cite{Dow18}. For another overview of knot Floer homology, we direct the reader to \cite{Hom21}. For greater detail, see \cite{OS04}.

All of our definitions will be made over $\Q$ since, as in \cite{Dow18, Dow24}, the tools needed depend on working with a field of characteristic zero (we need $n$, and also $2$, for the spectral sequence of \cite{Dow24}, to be invertible), at the cost of keeping track of signs. 
However, most definitions here can be adapted to one's preferred field for other applications. Knot Floer complexes come in a variety of flavors, but we are only interested in two: the minus flavor $\CFK^-(L)$ and the so-called master complex $\CFK_{U,V}(L)$.

Given a multi-pointed unpunctured Heegaard diagram $\mathcal{H}$ for a null-homologous link $L\subset Y$, the so-called \emph{master complex} $\CFK_{U,V}(\mathcal{H})$ is the free $\Q[U_1,\dots,U_k,V_1,\dots,V_k]$-module generated by $\mathbb{T}_\alpha\cap\mathbb{T}_\beta$, the intersection points of the totally real Lagrangian tori $\mathbb{T}_\alpha$ and $\mathbb{T}_\beta$ in the symplectic manifold $\mathrm{Sym}^{g+k-1}(\Sigma)$ (or $\mathrm{Sym}^{g+k}(\Sigma)$ for punctured diagrams), given by the symmetric product of the $\boldsymbol{\alpha}$ (and $\boldsymbol{\beta}$, for $\mathbb{T}_\beta$) curves on $\mathcal{H}$. The differential in this complex counts the rigid, unparametrized Maslov index 1 pseudo-holomorphic disks from a generator $x$ to $y$, keeping track of which basepoints the disks cross and their multiplicities:
\[
\partial_{U,V}(x)=\sum_{y\in\mathbb{T}_\alpha\bigcap\mathbb{T}_\beta}\ \sum_{\substack{\phi\in\pi_2(x,y)\\ \mu(\phi)=1}}\#\widehat{\mathcal{M}}(\phi)U_1^{n_{w_1}(\phi)}\cdots U_k^{n_{w_k}(\phi)}V_1^{n_{z_1}(\phi)}\cdots V_k^{n_{z_k}(\phi)}y.
\]

\begin{remark}\label{rem:curved}
\cite[Lemma 2.4]{Dow18} In general, $\CFK_{U,V}(\mathcal{H})$ is not a true chain complex, but a curved one. Let $w_{a(i)}$ be the $\mathbf{w}$ basepoint connected to $z_i$ via an arc in the alpha handlebody, and $w_{b(i)}$ be the one connected to $z_i$ via an arc in the beta handlebody. Then the curved complex has potential
\[
\partial_{U,V}^2=\sum_{i=1}^k(U_{a(i)}-U_{b(i)})V_i.
\]
When each component of $L$ has exactly one $\mathbf{w}$ and $\mathbf{z}$ basepoint, $\CFK_{U,V}(\mathcal{H})$ is a true complex.
\end{remark}

The complex $\CFK^-(\mathcal{H})$ is similarly defined as the free $\Q[U_1,\dots,U_k]$-module generated by $\mathbb{T}_\alpha\cap\mathbb{T}_\beta$. In the minus flavor, the differential counts disks which only pass through $\mathbf{w}$ basepoints:
\[
\partial^-(x)=\sum_{y\in\mathbb{T}_\alpha\bigcap\mathbb{T}_\beta}\ \sum_{\substack{\phi\in\pi_2(x,y)\\ \mu(\phi)=1\\ n_{\mathbf{z}}(\phi)=0}}\#\widehat{\mathcal{M}}(\phi)U_1^{n_{w_1}(\phi)}\cdots U_k^{n_{w_k}(\phi)}y.
\]

The definition of $\CFK_{U,V}(L)$ is the same for punctured Heegaard diagrams, where the differential doesn't count discs that go through $p$:
\[
\partial_{U,V}(x)=\sum_{y\in\mathbb{T}_\alpha\bigcap\mathbb{T}_\beta}\ \sum_{\substack{\phi\in\pi_2(x,y)\\ \mu(\phi)=1\\ n_p(\phi)=0}}\#\widehat{\mathcal{M}}(\phi)U_1^{n_{w_1}(\phi)}\cdots U_k^{n_{w_k}(\phi)}V_1^{n_{z_1}(\phi)}\cdots V_k^{n_{z_k}(\phi)}y.
\]

The chain homotopy type of each complex is an invariant of the link $L$ (for the curved complex $\CFK_{U,V}$, see \cite{Zem19}), so we often ignore the Heegaard diagram dependence and write $\CFK_{U,V}(L)$ and $\CFK^-(L)$. The complex $\CFK^-(L)$ is a true chain complex, and the homology $\HFK^-(L)=H_*(\CFK^-(L))$ is the familiar minus flavor of knot Floer homology. We make the following remarks.

\begin{remark}\label{rem:thin}
The thin knots considered in this note enjoy another noteworthy trait: $\CFK_{U,V}(K)$ has a model in which every arrow has length one, that is, every arrow is labeled $U$ or $V$. This is discussed at length in \cite{Pet13} and presented quickly in Section~\ref{sec:proof} (Theorem~\ref{thm:thin-structure}). This fails for $L$-space knots in general: $\CFK_{U,V}(T_{3,4})$ has arrows of length two (Example~\ref{ex:T34}).
\end{remark}

\begin{remark}
Since, in our setting, the manifold $Y$ will be $S^3$, we needn't worry about the splitting of these theories over $\mathrm{Spin}^c$ structures. There is also a great deal of grading information in these invariants, but it isn't needed for the main theorem, so we suppress it.
\end{remark}

\subsection{The Family $\HFK_n$}
In \cite{Dow18}, Dowlin constructs a family of link invariants designed to mimic features of the $\mathfrak{sl}_n$-homology theories of Khovanov and Rozansky \cite{KR08a}.

\begin{definition}\label{def:CFKn}
\cite[Definition 2.5]{Dow18} Given a multi-pointed punctured Heegaard diagram $\mathcal{H}$ for a link $L$ in a 3-manifold $Y$, define the chain complex $\CFK_n(\mathcal{H})$ to be the quotient
\[
\CFK_n(\mathcal{H})=\CFK_{U,V}(\mathcal{H})\Bigg/\left(V_i-\frac{U_{a(i)}^n-U_{b(i)}^n}{U_{a(i)}-U_{b(i)}}\right)
\]
where $a(i)$ and $b(i)$ are as in~\ref{rem:curved}, the difference quotient denotes the polynomial $\sum_{j=0}^{n-1}U_{a(i)}^jU_{b(i)}^{n-1-j}$, and the differential $\partial_n$ is the induced differential from $\partial_{U,V}$.
\end{definition}

\begin{remark}\label{rem:onevariable}
In the case where $L$ is a knot, we may only have one variable, that is, $a(1)=b(1)$. The difference quotient is then $\sum_{j=0}^{n-1}U^jU^{n-1-j}=nU^{n-1}$, which is the derivative
\[
V=\frac{d}{dU}\big(U^n\big)=nU^{n-1}.
\]
\end{remark}

\begin{theorem}\label{thm:dsquared}
\cite[Lemma 2.6]{Dow18} The map $\partial_n:\CFK_n(\mathcal{H})\to\CFK_n(\mathcal{H})$ satisfies $\partial_n^2=0$.
\end{theorem}

Let $\HFK_n(\mathcal{H})=H_*(\CFK_n(\mathcal{H}),\partial_n)$. With the following theorem, we can define the link invariants.

\begin{theorem}\label{thm:invariance}
\cite[Theorem 2.7]{Dow18} If $\mathcal{H}_1$ and $\mathcal{H}_2$ are two punctured Heegaard diagrams for a null-homologous link $L$ with $\ell$ components, then $\CFK_n(\mathcal{H}_1)$ and $\CFK_n(\mathcal{H}_2)$ are chain homotopy equivalent as complexes of graded $\Q[U_1,\dots,U_\ell]$-modules.
\end{theorem}

\begin{definition}
\cite[Definition 2.8]{Dow18} Define $\HFK_n(L)=H_*(\CFK_n(\mathcal{H}))$ where $\mathcal{H}$ is any punctured Heegaard diagram for $L$.
\end{definition}

We perform some computations in Section~\ref{sec:examples}. These examples illustrate that, for knots, one has a simpler equivalent definition of $\HFK_n(K)$. In fact, in the case where $L=K$ is a knot, $\HFK_n(K)$ can be computed from $\CFK_{U,V}(K)$ without a punctured Heegaard diagram.

\begin{theorem}\label{thm:knot-quotient}
\cite[Lemma 2.16]{Dow18} If $\mathcal{H}$ is an unpunctured diagram for a knot $K$ in $S^3$ with a single pair of basepoints $w$ and $z$, then
\[
\HFK_n(K)=H_*\left(\frac{\CFK_{U,V}(\mathcal{H})}{(UV,V-nU^{n-1})}\right).
\]
\end{theorem}

\begin{remark}\label{rem:n=1}
For $n=1$, the relations in Theorem~\ref{thm:knot-quotient} become $U=0$ and $V=1$, so the quotient complex is $\widehat{CF}(S^3)$ and $\HFK_1(K)\cong\widehat{HF}(S^3)\cong\Q$ for every knot $K$ (cf.\ \cite{Dow18}). For this reason we assume $n\geq2$ whenever we compare $\HFK_n(K)$ with $\HFK^-(K)$.
\end{remark}

\section{Computational Evidence}\label{sec:examples}

\begin{figure}[ht]
\[
\begin{tikzcd}
\Q[U] & \Q[U] \arrow[l, "U"'] \arrow[d, "0"] \\
 & \Q[U]
\end{tikzcd}
\qquad\qquad
\begin{tikzcd}
\frac{\Q[U]}{(U^n)} & \frac{\Q[U]}{(U^n)} \arrow[l, "U"'] \arrow[d, "nU^{n-1}"] \\
 & \frac{\Q[U]}{(U^n)}
\end{tikzcd}
\]
\caption{On the left: the complex $\CFK^-(T_{2,3})$. On the right: the complex $\CFK_n(T_{2,3})$.}
\label{fig:trefoil}
\end{figure}

The master complex $\CFK_{U,V}(L)$ is well-known for torus knots \cite{OS05a}. An application of Theorem~\ref{thm:knot-quotient} simplifies the computations.

\begin{remark}\label{rem:ring}
Since $n$ is invertible in $\Q$, the map $\Q[U,V]\to\frac{\Q[U]}{(U^n)}$ sending $U\mapsto U$ and $V\mapsto nU^{n-1}$ induces a ring isomorphism $\frac{\Q[U,V]}{(UV,V-nU^{n-1})}\cong\frac{\Q[U]}{(U^n)}$: indeed $\frac{\Q[U,V]}{(V-nU^{n-1})}\cong\Q[U]$, and $UV$ maps to $nU^n$. In the following computations, we write $\frac{\Q[U]}{(U^n)}$ as the summand of the complexes $\CFK_n(K)$, with $V$ acting by $nU^{n-1}$. This module is not isomorphic, as a $\Q[U,V]$-module, to $\frac{\Q[U,V]}{(V,U^n)}$, on which $V$ acts by zero, although the two are isomorphic as $\Q[U]$-modules.
\end{remark}

\begin{example}\label{ex:trefoil}
The Trefoil.

As mentioned above, $\CFK_{U,V}(T_{p,q})$ can be read off from its Alexander polynomial. For the trefoil, it is well-known that $\Delta(T_{2,3})=t^{-1}-1+t$. Since the spacing between the powers of $t$ are only one each, we have only a single power of $U$ and $V$ as the maps on the complex, yielding the usual staircase picture:
\[
\begin{tikzcd}
\Q[U,V] & \Q[U,V] \arrow[l, "U"'] \arrow[d, "V"] \\
 & \Q[U,V]
\end{tikzcd}
\]
The complex $\CFK_n(T_{2,3})$ looks the same with the appropriate quotients. Both $\CFK^-(K)$ and $\CFK_n(K)$ are pictured in Figure~\ref{fig:trefoil}. On the level of homology, for $n\geq2$, we see the following:
\begin{align*}
\HFK_n(T_{2,3})&\cong\frac{\Q[U]}{(U^n)}\oplus\Q\oplus\Q\\
\HFK^-(T_{2,3})&\cong\Q[U]\oplus\Q.
\end{align*}
Notice that since $\Tor_1^{\Q[U]}(\Q\oplus\Q[U],\frac{\Q[U]}{(U^n)})\cong\Q$, by the universal coefficient theorem,
\begin{align*}
\HFK_n(T_{2,3})&\cong\HFK^-(T_{2,3})\tensQU\frac{\Q[U]}{(U^n)}\bigoplus\Tor_1^{\Q[U]}\Big(\HFK^-(T_{2,3});\frac{\Q[U]}{(U^n)}\Big)\\
&\cong(\Q[U]\oplus\Q)\tensQU\frac{\Q[U]}{(U^n)}\bigoplus\Q\\
&\cong\frac{\Q[U]}{(U^n)}\oplus\Q\oplus\Q.
\end{align*}
(For $n=1$ the answer is different: $\HFK_1(T_{2,3})\cong\Q$ by Remark~\ref{rem:n=1}.) We see similar behavior for a knot whose master complex has boxes and is not a torus knot.
\end{example}

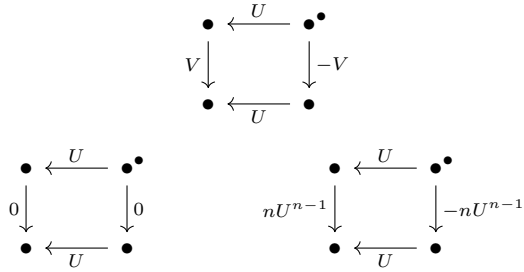
\begin{figure}[ht]
\[
\begin{tikzcd}
\bullet \arrow[d, "V"'] & \bullet\rlap{$^{\bullet}$} \arrow[l, "U"'] \arrow[d, "-V"] \\
\bullet & \bullet \arrow[l, "U"]
\end{tikzcd}
\]
\[
\begin{tikzcd}
\bullet \arrow[d, "0"'] & \bullet\rlap{$^{\bullet}$} \arrow[l, "U"'] \arrow[d, "0"] \\
\bullet & \bullet \arrow[l, "U"]
\end{tikzcd}
\qquad\qquad
\begin{tikzcd}
\bullet \arrow[d, "nU^{n-1}"'] & \bullet\rlap{$^{\bullet}$} \arrow[l, "U"'] \arrow[d, "-nU^{n-1}"] \\
\bullet & \bullet \arrow[l, "U"]
\end{tikzcd}
\]
\caption{On the top: the complex $\CFK_{U,V}(4_1)$, where each dot represents the module $\Q[U,V]$. On the left: the complex $\CFK^-(4_1)$, where each dot represents the module $\Q[U]$. On the right: the complex $\CFK_n(4_1)$, where each dot represents the module $\frac{\Q[U]}{(U^n)}$.}
\label{fig:fig8}
\end{figure}

\begin{example}\label{ex:fig8}
The Figure-Eight.

We first look at the master complex for $4_1$, and then its quotients, as shown in Figure~\ref{fig:fig8}. The invariants for $4_1$ are, for $n\geq2$,
\begin{align*}
\HFK_n(4_1)&\cong\Q\oplus\Q\oplus\Q\oplus\Q\oplus\frac{\Q[U]}{(U^n)}\\
\HFK^-(4_1)&\cong\Q[U]\oplus\Q\oplus\Q.
\end{align*}
Again, the universal coefficient theorem gives us the following:
\begin{align*}
\HFK_n(4_1)&\cong\HFK^-(4_1)\tensQU\frac{\Q[U]}{(U^n)}\bigoplus\Tor_1^{\Q[U]}\Big(\HFK^-(4_1);\frac{\Q[U]}{(U^n)}\Big)\\
&\cong\frac{\Q[U]}{(U^n)}\oplus\Q\oplus\Q\oplus\Q\oplus\Q.
\end{align*}
\end{example}

\begin{remark}
While one might wish to explore examples with diagonal arrows, we note that all diagonal arrows in the master complex for a knot disappear when descending to any of the complexes at hand.
\end{remark}

\begin{example}\label{ex:T34}
The Torus Knot $T_{3,4}$.

Here $\Delta(T_{3,4})=t^{-3}-t^{-2}+1-t^{2}+t^{3}$, whose gaps are $1,2,2,1$, so $\CFK_{U,V}(T_{3,4})$ is the complex
\[
\begin{tikzcd}
x_1 & x_2 \arrow[l, "U"'] \arrow[d, "V^2"] & \\
 & x_3 & x_4 \arrow[l, "U^2"'] \arrow[d, "V"] \\
 & & x_5
\end{tikzcd}
\]
Take $n=2$. Since $V$ acts by $2U$ and $U^2=0$, in $\CFK_2(T_{3,4})$ we have $\partial_2x_2=Ux_1$ and $\partial_2x_4=2Ux_5$, while in $\CFK^-(T_{3,4})\otimes\frac{\Q[U]}{(U^2)}$ we have $\partial x_2=Ux_1$ and $\partial x_4=0$. Hence
\[
\HFK_2(T_{3,4})\cong\frac{\Q[U]}{(U^2)}\oplus\Q^4,
\]
of dimension $6$. On the other hand $\HFK^-(T_{3,4})\cong\Q[U]\oplus\frac{\Q[U]}{(U)}\oplus\frac{\Q[U]}{(U^2)}$, and the universal coefficient theorem gives
\[
\HFK^-(T_{3,4})\tensQU\frac{\Q[U]}{(U^2)}\bigoplus\Tor_1^{\Q[U]}\Big(\HFK^-(T_{3,4});\frac{\Q[U]}{(U^2)}\Big)\cong\Big(\frac{\Q[U]}{(U^2)}\Big)^3\oplus\Q^2,
\]
of dimension $8$. The value $6$ agrees with Dowlin's computation $\dim_\Q\HFK_2(T_{3,3k+1})=2+4k$ \cite[Lemma 5.9]{Dow18}. The discrepancy comes from the arrow $x_4\to 2Ux_5$: a vertical arrow of length one survives in $\CFK_2$, and there is no horizontal arrow of length at most $n-1=1$ leaving $x_4$ to absorb it (compare Lemma~\ref{lem:staircase} and Proposition~\ref{prop:sharp}). For $n\geq3$, both sides have dimension $n+6$.
\end{example}

\section{Proof of Main Theorem}\label{sec:proof}

Motivated by the preceding examples, we prove Theorem~\ref{thm:main-intro} using the universal coefficient theorem. In fact, since $\Q[U]$ is a principal ideal domain and $\CFK^-(K)$ is a finitely generated free complex over $\Q[U]$ \cite{Spa66} (see also \cite[Theorem 6.1]{HN13}),
\[
H_*\Big(\CFK^-(K)\tensQU\frac{\Q[U]}{(U^n)}\Big)\cong\Big[\HFK^-(K)\tensQU\frac{\Q[U]}{(U^n)}\Big]\bigoplus\Tor_1^{\Q[U]}\Big(\HFK^-(K);\frac{\Q[U]}{(U^n)}\Big).
\]
Here and below these are isomorphisms of ungraded $\Q[U]$-modules: the splitting is not natural, and in a graded statement the $\Tor$ summand appears with a shift in degree.

Recall that, by Theorem~\ref{thm:knot-quotient}, Dowlin's complex $\CFK_n(K)$ for knots may be replaced by the quotient
\[
\CFK_n(K)\simeq\frac{\CFK_{U,V}(K)}{(UV,V-nU^{n-1})}.
\]
In Examples \ref{ex:trefoil} and \ref{ex:fig8} in Section~\ref{sec:examples}, for $n\geq2$, this complex turns out to be isomorphic to
\[
\CFK^-(K)\tensQU\frac{\Q[U]}{(U^n)}
\]
which we can compute using the universal coefficient theorem. This motivates us to propose another definition:

\begin{definition}\label{def:Cn}
Given a finitely generated free chain complex $C=\bigoplus\Q[U,V]$ over $\Q[U,V]$, define
\begin{align*}
C_n&:=\frac{C}{(UV,V-nU^{n-1})}\\
\widetilde{C}_n&:=\frac{C}{(V,U^n)}.
\end{align*}
\end{definition}

\begin{remark}\label{rem:tensor}
Notice that
\[
\frac{C}{(UV,V-nU^{n-1})}\cong C\underset{\Q[U,V]}{\otimes}\frac{\Q[U,V]}{(UV,V-nU^{n-1})},
\]
and similarly that
\[
\frac{C}{(V,U^n)}\cong C\underset{\Q[U,V]}{\otimes}\frac{\Q[U,V]}{(V,U^n)}\cong\frac{C}{(V)}\tensQU\frac{\Q[U]}{(U^n)},
\]
so $\CFK_n(K)$ and $\widetilde{\CFK}_n(K):=\CFK_{U,V}(K)/(V,U^n)\cong\CFK^-(K)\otimes\frac{\Q[U]}{(U^n)}$ are as in Definition~\ref{def:Cn}. 
\end{remark}

In general, $C_n$ and $\widetilde{C}_n$ are not isomorphic as $\Q[U,V]$-modules or quasi-isomorphic, while their underlying $\Q[U]$-modules are isomorphic (Remark~\ref{rem:ring}). In fact, we have the following counterexample:

\begin{lemma}\label{lem:counterexample}
$C_n$ and $\widetilde{C}_n$ do not always have the same homology.
\end{lemma}

\begin{proof}
Consider the complex $C$ given by
\[
\begin{tikzcd}
\Q[U,V] \arrow[r, "V"] & \Q[U,V].
\end{tikzcd}
\]
Then $C_n$ is
\[
\begin{tikzcd}
\frac{\Q[U]}{(U^n)} \arrow[r, "nU^{n-1}"] & \frac{\Q[U]}{(U^n)}
\end{tikzcd}
\]
and $\widetilde{C}_n$ is
\[
\begin{tikzcd}
\frac{\Q[U]}{(U^n)} \arrow[r, "0"] & \frac{\Q[U]}{(U^n)}.
\end{tikzcd}
\]
On the level of homology,
\begin{align*}
H_*(C_n)&\cong\frac{\Q[U]}{(U^{n-1})}\oplus\frac{\Q[U]}{(U^{n-1})}\\
H_*(\widetilde{C}_n)&\cong\frac{\Q[U]}{(U^n)}\oplus\frac{\Q[U]}{(U^n)}.
\end{align*}
\end{proof}

\begin{figure}[ht]
\[
\begin{tikzcd}[column sep=3em, row sep=1.8em]
x_1 & x_2 \arrow[l, "U^{k_1}"'] \arrow[d, "V^{k_2}"'] & & \\
 & x_3 & {} \arrow[l, "U^{k_3}"'] & \\
 & & \ddots & \\
 & & {} & x_{m-1} \arrow[l, "U^{k_{m-2}}"'] \arrow[d, "V^{k_{m-1}}"] \\
 & & & x_m
\end{tikzcd}
\qquad\qquad
\begin{tikzcd}[column sep=3em, row sep=1.8em]
x_1 \arrow[d, "V^{k_1}"'] & & \\
x_2 & x_3 \arrow[l, "U^{k_2}"'] & \\
 & \vdots \arrow[d, "V^{k_{m-2}}"'] & \\
 & x_{m-1} & x_m \arrow[l, "U^{k_{m-1}}"']
\end{tikzcd}
\]
\caption{The two possible staircase complexes: the left form and the right form.}
\label{fig:staircases}
\end{figure}
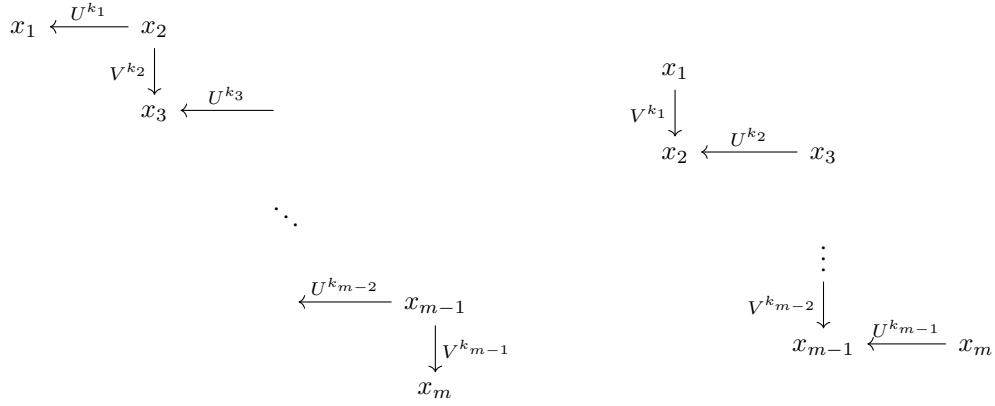

Although a general complex over $\Q[U,V]$ will not lead to isomorphic specializations, we observe that the complexes are chain homotopy equivalent in many important cases.

\subsection{$L$-space and thin knots}
The aforementioned classes of knots of interest are $L$-space knots and thin knots, as defined in Section~1; throughout, thinness refers to $\widehat{\HFK}(K;\Q)$. We now turn to defining the structures present in the knot Floer complexes of these kinds of knots.

\begin{definition}[Box]\label{def:box}
A chain complex $C=\bigoplus_{i=1}^4\Q[U,V]_{<x_i>}$ is a \emph{box complex} if it is in the form
\[
\begin{tikzcd}
x_2 \arrow[d, "V^l"'] & x_1 \arrow[l, "U^k"'] \arrow[d, "-V^l"] \\
x_4 & x_3 \arrow[l, "U^k"]
\end{tikzcd}
\]
for some positive integers $k$ and $l$, where each $x_i$ is the generator of $\Q[U,V]$ corresponding to the $i$-th summand of $C$.
\end{definition}

\begin{definition}[Staircase]\label{def:staircase}
A chain complex $C=\bigoplus_{i=1}^m\Q[U,V]_{<x_i>}$ is a \emph{staircase complex} if it is in one of the forms in Figure~\ref{fig:staircases} for some positive integers $k_i$ and odd integer $m\geqslant1$, where each $x_i$ is the generator of $\Q[U,V]$ corresponding to the $i$-th summand of $C$. In both forms the arrow between $x_i$ and $x_{i+1}$ has length $k_i$; in the left form the sources are the $x_i$ with $i$ even, and in the right form the $x_i$ with $i$ odd. Thus, in either form, for each source $x_i$,
\[
\partial x_i=U^{k_{i-1}}x_{i-1}+V^{k_i}x_{i+1},
\]
where terms involving $x_0$ or $x_{m+1}$ are omitted. When $m=1$, $C=\Q[U,V]_{<x_1>}$ with zero differential.
\end{definition}

\begin{example}
The master complex for the trefoil, $\CFK_{U,V}(T_{2,3})$ as in Example~\ref{ex:trefoil}, is a staircase complex.
\end{example}

\begin{figure}[ht]
\[
\begin{tikzcd}[row sep=small, column sep=large]
x_1 & x_2 \arrow[dl] \arrow[ddddr] & x_3 \\
Ux_1 & Ux_2 \arrow[dl] & Ux_3 \\
U^2x_1 & U^2x_2 & U^2x_3 \\
\vdots & \vdots & \vdots \\
U^{n-1}x_1 & U^{n-1}x_2 & U^{n-1}x_3
\end{tikzcd}
\]
\caption{The tower (vector space) structure of the complex $\CFK_n(T_{2,3})$ (the complex shown in Figure~\ref{fig:trefoil}). For example, $\partial_n(x_2)=Ux_1+nU^{n-1}x_3$ and $\partial_n(Ux_2)=U^2x_1$.}
\label{fig:towers}
\end{figure}
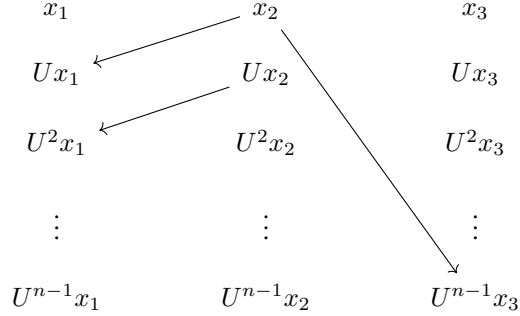

\begin{remark}\label{rem:towers}
Given a chain complex $C$ over $\Q[U,V]$, we often have Figure~\ref{fig:towers} in mind for the complexes $C_n$ and $\widetilde{C}_n$. Each generator corresponding to a summand has a tower of chains, and the differential is marked by arrows between them. Figure~\ref{fig:towers} shows $\CFK_n(T_{2,3})$. This structure also clearly shows $C$ as a $\Q$-vector space.
\end{remark}

With these definitions in place, we appeal to two structure theorems concerning the master complex $\CFK_{U,V}(K)$ of thin and $L$-space knots. Both are statements up to homotopy equivalence over $\Q[U,V]$. The proofs for these and Theorem \ref{thm:main} rely on the following Lemma:

\begin{lemma}[Cancellation lemma]\label{lem:cancellation}
\cite[Lemma 4.1]{HN13} Let $(C,\partial)$ be a chain complex of $\mathcal{R}$-modules, freely generated by chains $\{x_i\}$, and suppose that $d(x_k,x_l)$ is a unit in $\mathcal{R}$, where
\[
\partial(x_k)=\sum_{m\neq l}a_mx_m+d(x_k,x_l)x_l.
\]
Then we can define a complex $(C',\partial')$, freely generated by $\{x_i\mid i\neq k,\ i\neq l\}$, which is chain homotopy equivalent to $(C,\partial)$. Writing $d(x_i,x_j)$ for the coefficient of $x_j$ in $\partial x_i$, the differential is given by $d'(x_i,x_j)=d(x_i,x_j)-d(x_i,x_l)\,d(x_k,x_l)^{-1}\,d(x_k,x_j)$.
\end{lemma}

\begin{theorem}\label{thm:thin-structure}
(cf.\ \cite[Lemma 7]{Pet13}) If a knot $K$ is Floer homologically thin, then $\CFK_{U,V}(K)$ is homotopy equivalent over $\Q[U,V]$ to a direct sum of one staircase complex with all $k_i=1$ and $m=2|\tau(K)|+1$, of the left form in Figure~\ref{fig:staircases} if $\tau(K)\geq0$ and of the right form if $\tau(K)<0$, and box complexes with $k=l=1$.
\end{theorem}

\begin{proof}

Here we need the gradings that were suppressed in Section~\ref{sec:background}. The master complex $\CFK_{U,V}(K)$ is bigraded by $(\mathrm{gr}_U,\mathrm{gr}_V)$, in which $U$ has bidegree $(-2,0)$, $V$ has bidegree $(0,-2)$ and the differential has bidegree $(-1,-1)$. The Alexander and Maslov gradings are $A=\frac{1}{2}(\mathrm{gr}_U-\mathrm{gr}_V)$ and $M=\mathrm{gr}_U$ \cite{Zem19,Hom21}.
 
\emph{Step 1: every arrow has length one.} Applying Lemma~\ref{lem:cancellation} repeatedly to arrows whose coefficients are invertible, we obtain a complex $\mathcal{C}$, homotopy equivalent to $\CFK_{U,V}(K)$ over $\Q[U,V]$, which is free on a homogeneous basis and satisfies $\partial\mathcal{C}\subset(U,V)\mathcal{C}$. The process terminates since there are finitely many generators, and it preserves the bigrading. Let $\Omega$ be the $\Q$-span of this basis, graded by $A$ as $\Omega=\bigoplus_a\Omega_a$. Since the differential of $\mathcal{C}/(U,V)$ vanishes, $\Omega\cong\widehat{\HFK}(K;\Q)$ as graded vector spaces. Since $K$ is thin, $A-M=-\frac{1}{2}(\mathrm{gr}_U+\mathrm{gr}_V)$ is constant on $\Omega$. If $cU^iV^jy$ with $c\neq0$ appears in $\partial x$ for basis elements $x$ and $y$, then comparing $\mathrm{gr}_U+\mathrm{gr}_V$ on both sides gives
\[
(\mathrm{gr}_U+\mathrm{gr}_V)(y)-2(i+j)=(\mathrm{gr}_U+\mathrm{gr}_V)(x)-2,
\]
so $i+j=1$. Hence there are $\Q$-linear maps $\xi,\eta:\Omega\to\Omega$ such that
\[
\partial=U\xi+V\eta
\]
on $\mathcal{C}=\Omega\otimes\Q[U,V]$, with $\xi(\Omega_a)\subset\Omega_{a+1}$ and $\eta(\Omega_a)\subset\Omega_{a-1}$.
 
\emph{Step 2: an exterior algebra.} Expanding $\partial^2=U^2\xi^2+UV(\xi\eta+\eta\xi)+V^2\eta^2=0$ gives
\[
\xi^2=0,\qquad\eta^2=0,\qquad\xi\eta=-\eta\xi.
\]
So $\Omega$ is a graded module over the exterior algebra $\Lambda=\Q\langle\xi,\eta\rangle/(\xi^2,\eta^2,\xi\eta+\eta\xi)$, and a decomposition $\Omega=\Omega'\oplus\Omega''$ into graded $\Lambda$-submodules gives a decomposition into subcomplexes
\[
\mathcal{C}=(\Omega'\otimes\Q[U,V])\oplus(\Omega''\otimes\Q[U,V]).
\] Moreover, the complex $(\Omega,\xi)$ is $\mathcal{C}\otimes\frac{\Q[U,V]}{(U-1,V)}\simeq\CFK^-(K)\otimes_{\Q[U]}\frac{\Q[U]}{(U-1)}$, whose homology is $\widehat{HF}(S^3)\cong\Q$ \cite{OS04}. In the same way $(\Omega,\eta)$ computes $\widehat{HF}(S^3)$ using the basepoint $w$. Thus
\[
\dim_\Q H_*(\Omega,\xi)=\dim_\Q H_*(\Omega,\eta)=1.
\]
 
\emph{Step 3: splitting off boxes.} Suppose $\xi\eta x\neq0$ for some $x\in\Omega_a$. Then $x$, $\xi x$, $\eta x$ and $\xi\eta x$ are linearly independent: $\xi x$ and $\eta x$ are nonzero and lie in $\Omega_{a+1}$ and $\Omega_{a-1}$, and $\xi\eta x=cx$ is impossible, since it would give $c^2x=(\xi\eta)^2x=-\xi^2\eta^2x=0$. So their span $\Gamma$ is a free $\Lambda$-submodule of rank one. Choose a linear map $\lambda:\Omega\to\Q$ which vanishes outside $\Omega_a$ with $\lambda(x)=0$ and $\lambda(\xi\eta x)=1$, and define $\rho:\Omega\to\Gamma$ by
\[
\rho(w)=\lambda(\xi\eta w)\,x-\lambda(\eta w)\,\xi x+\lambda(\xi w)\,\eta x+\lambda(w)\,\xi\eta x.
\]
A direct check using the relations of Step~2 shows that $\rho\xi=\xi\rho$ and $\rho\eta=\eta\rho$; since $\lambda$ is supported on $\Omega_a$, $\rho$ preserves the grading; and $\rho(x)=x$, so $\rho$ is the identity on $\Gamma$. Hence $\Omega=\Gamma\oplus\ker\rho$ as graded $\Lambda$-modules. In the basis
\[
x_1=x,\qquad x_2=\xi x,\qquad x_3=-\eta x,\qquad x_4=\eta\xi x,
\]
the complex $\Gamma\otimes\Q[U,V]$ is the box complex of Definition~\ref{def:box} with $k=l=1$: $\partial x_1=Ux_2-Vx_3$, $\partial x_2=Vx_4$ and $\partial x_3=Ux_4$. By induction on $\dim_\Q\Omega$,
\[
\Omega=\Gamma_1\oplus\cdots\oplus\Gamma_b\oplus\Omega',
\]
where each $\Gamma_i\otimes\Q[U,V]$ is a box complex with $k=l=1$ and $\xi\eta=\eta\xi=0$ on $\Omega'$.
 
\emph{Step 4: staircases.} Let $\Theta=\ker\xi\cap\ker\eta\subset\Omega'$, and let $\Psi$ be a graded complement of $\Theta$ in $\Omega'$. Since all products of $\xi$ and $\eta$ vanish on $\Omega'$, both maps send $\Omega'$ into $\Theta$ and vanish on $\Theta$. So $\Omega'$ is determined by the restrictions $\xi:\Psi_a\to\Theta_{a+1}$ and $\eta:\Psi_a\to\Theta_{a-1}$, which form a representation of the quiver with vertices $\Psi_a$ and $\Theta_a$ for $a\in\mathbb{Z}$ and these arrows; a decomposition of this representation is a decomposition of $\Omega'$ into graded $\Lambda$-submodules. The underlying graph of the quiver is a disjoint union of two lines
\[
\cdots\ \text{---}\ \Theta_{a-1}\ \text{---}\ \Psi_a\ \text{---}\ \Theta_{a+1}\ \text{---}\ \Psi_{a+2}\ \text{---}\ \cdots,
\]
and since $\Omega'$ is finite-dimensional, the representation is supported on a quiver of type $A$. By Gabriel's theorem \cite{Gab72}, it is a direct sum of interval representations, which are one-dimensional at each vertex of a segment of a line, zero elsewhere, and have every arrow in the segment an isomorphism. After rescaling the basis vectors along the segment, each such summand $\Upsilon$ gives a staircase complex $\Upsilon\otimes\Q[U,V]$ with all $k_i=1$, whose sources are the basis vectors lying in $\Psi$. It is of the left form if the endpoints of the segment lie in $\Theta$ (which includes the case $m=1$), and of the right form if they lie in $\Psi$.
 
\emph{Step 5: exactly one staircase.} The complexes $(\Gamma_i,\xi)$ and $(\Gamma_i,\eta)$ are acyclic. If an interval summand $\Upsilon$ has $v$ basis vectors, it has $v-1$ arrows, say $e_\xi$ of them given by $\xi$ and $e_\eta$ by $\eta$, each an isomorphism between one-dimensional spaces. Hence
\[
\dim_\Q H_*(\Upsilon,\xi)+\dim_\Q H_*(\Upsilon,\eta)=(v-2e_\xi)+(v-2e_\eta)=2.
\]
By Step~2 the total of these quantities over $\Omega$ is $2$, so there is exactly one interval summand $\Upsilon$, and $\dim_\Q H_*(\Upsilon,\xi)=\dim_\Q H_*(\Upsilon,\eta)=1$ forces $e_\xi=e_\eta$, so $m=v=2e_\xi+1$ is odd. Thus
\[
\mathcal{C}=\mathcal{S}\oplus\mathcal{B}_1\oplus\cdots\oplus\mathcal{B}_b,
\]
where $\mathcal{S}=\Upsilon\otimes\Q[U,V]$ is a staircase complex with all $k_i=1$ and each $\mathcal{B}_i=\Gamma_i\otimes\Q[U,V]$ is a box complex with $k=l=1$.
 
\emph{Step 6: the length of the staircase.} We use the convention \cite{Hom21}
\[
\tau(K)=-\max\{A(y):y\in\HFK^-(K)\text{ homogeneous},\ U^dy\neq0\text{ for all }d\}.
\] Number the generators of $\mathcal{S}$ as in Definition~\ref{def:staircase}. Since $\xi$ raises and $\eta$ lowers $A$ by one, in either form $A(x_j)=A(x_1)-(j-1)$. Each $\mathcal{B}_i/(V)$ has homology $\Q_{<x_2>}\oplus\Q_{<x_4>}$, which is $U$-torsion, so the $U$-nontorsion part of $\HFK^-(K)=H_*(\mathcal{C}/(V))$ comes from $\mathcal{S}/(V)$, where it is generated by $x_m$ in the left form and by $x_1$ in the right form; this generator has Alexander grading $-\tau(K)$. The symmetry of the master complex which exchanges the roles of $w$ and $z$ \cite[Section 3.5]{OS04} interchanges $U$ and $V$ and replaces $A$ by $-A$. So the $V$-nontorsion part of $H_*(\mathcal{C}/(U))$ lies in Alexander grading $\tau(K)$, and it is generated by $x_1$ in the left form and by $x_m$ in the right form. In the left form, $2\tau(K)=A(x_1)-A(x_m)=m-1$; in the right form, $-2\tau(K)=A(x_1)-A(x_m)=m-1$. Hence $m=2|\tau(K)|+1$, the staircase is of the left form if $\tau(K)\geq0$, and of the right form if $\tau(K)<0$.

\end{proof}

\begin{example}[The knot $5_2$]
We illustrate the proof with the knot $5_2$, which is alternating and hence thin, with $\Delta_{5_2}(t)=2t-3+2t^{-1}$ and $\det(5_2)=7$. Choose its chirality so that $\tau(5_2)=1$. Then $\widehat{\HFK}(5_2;\Q)$ has rank $2$, $3$ and $2$ in Alexander gradings $1$, $0$ and $-1$, with $M=A-1$. Consider the reduced complex $\mathcal{C}$ with basis $y_1,y_2$ in Alexander grading $1$, $y_3,y_4,y_5$ in grading $0$, and $y_6,y_7$ in grading $-1$, and differential
\[
\partial y_3=Uy_1+Vy_7,\qquad\partial y_4=Uy_2-Vy_6+2Vy_7,\qquad\partial y_2=Vy_5,\qquad\partial y_6=Uy_5,
\]
with $\partial y_1=\partial y_5=\partial y_7=0$. We have $\CFK_{U,V}(5_2)$, written in a basis in which its decomposition is not clear, so we follow the proof to recover it. Here,
\[
\xi(y_3)=y_1,\quad\xi(y_4)=y_2,\quad\xi(y_6)=y_5,\qquad\eta(y_2)=y_5,\quad\eta(y_3)=y_7,\quad\eta(y_4)=-y_6+2y_7,
\]
and $\xi$ and $\eta$ vanish on the remaining basis vectors. One checks the relations of Step~2, and that $\xi$ and $\eta$ each have rank $3$, so $\dim_\Q H_*(\Omega,\xi)=\dim_\Q H_*(\Omega,\eta)=7-6=1$.
 
For Step~3, $\xi\eta(y_4)=\xi(-y_6+2y_7)=-y_5\neq0$, so take $x=y_4$ and let $\lambda$ be the linear map with $\lambda(y_5)=-1$ which vanishes on the other basis vectors. Then $\rho$ fixes $y_2$, $y_4$ and $y_5$, sends $y_6$ to $y_6-2y_7$, and sends $y_1$, $y_3$ and $y_7$ to zero. The box is spanned by
\[
x_1=y_4,\qquad x_2=y_2,\qquad x_3=y_6-2y_7,\qquad x_4=y_5,
\]
and indeed $\partial x_1=Ux_2-Vx_3$, $\partial x_2=Vx_4$ and $\partial x_3=Ux_4$. The change of basis $y_6\mapsto y_6-2y_7$ is what separates the box from the staircase.
 
For Step~4, $\Omega'=\ker\rho$ is spanned by $y_1$, $y_3$ and $y_7$, with $\Theta$ spanned by $y_1$ and $y_7$ and $\Psi$ spanned by $y_3$. The single segment is $\Theta_1\ \text{---}\ \Psi_0\ \text{---}\ \Theta_{-1}$, whose endpoints lie in $\Theta$, so it gives the staircase of the left form with $x_1=y_1$, $x_2=y_3$ and $x_3=y_7$, that is, $\partial x_2=Ux_1+Vx_3$. Thus
\[
\CFK_{U,V}(5_2)\simeq\CFK_{U,V}(T_{2,3})\oplus\mathcal{B},
\]
where $\mathcal{B}$ is a box complex with $k=l=1$. As in Step~6, the $U$-nontorsion generator of $\HFK^-(5_2)$ is $x_3=y_7$, in Alexander grading $-1=-\tau(5_2)$, and $m=3=2\tau(5_2)+1$. By Remark~\ref{rem:thin-formula}, $\HFK_n(5_2)\cong\frac{\Q[U]}{(U^n)}\oplus\Q^6$ for $n\geq2$.

\end{example}

\begin{theorem}\label{thm:lspace-structure}
\cite[Theorem 1.2]{OS05a}, \cite[Remark 6.6]{Hom14} If a knot $K$ is an $L$-space knot, then $\CFK_{U,V}(K)$ is homotopy equivalent over $\Q[U,V]$ to a single staircase complex of the left form in Figure~\ref{fig:staircases}, whose step lengths $k_1,\dots,k_{m-1}$ are the gaps of $\Delta_K$.
\end{theorem}

Here \cite[Lemma 7]{Pet13} proves the decomposition of Theorem~\ref{thm:thin-structure} over $\mathbb{F}_2$, for the Alexander-filtered complex $\CFK^-(K)$; the proof of Theorem~\ref{thm:thin-structure} above works over any field, and so recovers it. Similarly, \cite[Theorem 1.2]{OS05a} determines $\widehat{\HFK}(K)$ of an $L$-space knot, and the staircase form of the full complex is a standard consequence \cite[Remark 6.6]{Hom14}; over $\Q$, the signs of the arrows can be normalized by rescaling generators, since the underlying graph of a staircase is a path.

\subsection{Main theorem}
To prove the main theorem, we employ the following algebraic lemmata along with Lemma \ref{lem:cancellation}. The second term in $d'$ in Lemma \ref{lem:cancellation} is the ``zigzag'' through the cancelled pair; it is exactly what the change of basis in the first case of the proof of Proposition~\ref{prop:sharp} keeps track of.

\begin{lemma}\label{lem:box}
Let $n\geq2$ and let $C$ be a box complex. If $l\geq2$ or $k\leq n-1$, then $C_n\cong\widetilde{C}_n$ as complexes of $\frac{\Q[U]}{(U^n)}$-modules; in particular
\[
H_*(C_n)\cong H_*(\widetilde{C}_n)\cong\left(\frac{\Q[U]}{(U^{\min(k,n)})}\right)^4
\]
is an isomorphism of $\Q[U]$-modules. If $l=1$ and $k\geq n$, then
\[
H_*(C_n)\cong\left(\frac{\Q[U]}{(U^{n-1})}\right)^4,\qquad H_*(\widetilde{C}_n)\cong\left(\frac{\Q[U]}{(U^n)}\right)^4.
\]
\end{lemma}

\begin{proof}
Any box complex can be pictorally represented as in Definition~\ref{def:box}. Note that after the identifications $V=nU^{n-1}$ and $UV=0$ in the ground ring, if $l>1$ (here $n\geq2$ is used in the third line),
\begin{align*}
V^l&=V\cdot V^{l-1}\\
&=nU^{n-1}V^{l-1}\\
&=UV\cdot nU^{n-2}V^{l-2}\\
&=0.
\end{align*}

\begin{figure}[ht]
\[
\begin{tikzcd}
x_2 \arrow[d, "nU^{n-1}"'] & x_1 \arrow[l, "U^k"'] \arrow[d, "-nU^{n-1}"] \\
x_4 & x_3 \arrow[l, "U^k"]
\end{tikzcd}
\qquad\qquad
\begin{tikzcd}
x_2 \arrow[d, "0"'] & x_1 \arrow[l, "U^k"'] \arrow[d, "0"] \\
x_4 & x_3 \arrow[l, "U^k"]
\end{tikzcd}
\]
\caption{On the left: $C_n$ for $l=1$. On the right: $\widetilde{C}_n$ for $l=1$.}
\label{fig:box-l1}
\end{figure}

So we see that if $l>1$, the vertical arrows become zero in both $C_n$ and $\widetilde{C}_n$, and $C_n=\widetilde{C}_n$. Thus, we restrict our attention to when $l=1$. In this case, $C_n$ and $\widetilde{C}_n$ are as in Figure~\ref{fig:box-l1}. If $k\leq n-1$, put $x_2'=x_2-nU^{n-1-k}x_3$. Then
\[
\partial_nx_1=U^kx_2-nU^{n-1}x_3=U^kx_2',\qquad \partial_nx_2'=nU^{n-1}x_4-nU^{n-1-k}U^kx_4=0,
\]
so in the basis $\{x_1,x_2',x_3,x_4\}$ the differential of $C_n$ is that of $\widetilde{C}_n$. By direct computation using a similar diagram to Figure~\ref{fig:towers},
\[
H_*(C_n)\cong\left(\frac{\Q[U]}{(U^k)}\right)^4\cong H_*(\widetilde{C}_n).
\]
If $k\geq n$, then $U^k=0$ in the ground ring, so $C_n$ is the direct sum of two copies of the complex in the proof of Lemma~\ref{lem:counterexample} and $\widetilde{C}_n$ has zero differential; the stated homology follows from that computation.
\end{proof}

\begin{lemma}\label{lem:staircase}
Let $n\geq2$ and let $C$ be a staircase complex. Suppose that every vertical arrow of length one is accompanied by a horizontal arrow of length at most $n-1$ at its source (left form) or at its target (right form); that is,
\begin{enumerate}
\item if $C$ is of the left form, $k_{i-1}\leq n-1$ for every even $i$ with $k_i=1$;
\item if $C$ is of the right form, $k_{i+1}\leq n-1$ for every odd $i\leq m-2$ with $k_i=1$.
\end{enumerate}
Then $C_n\cong\widetilde{C}_n$ as complexes of $\frac{\Q[U]}{(U^n)}$-modules; in particular $H_*(C_n)\cong H_*(\widetilde{C}_n)$ is an isomorphism of $\Q[U]$-modules.
\end{lemma}

\begin{proof}
Any staircase complex can be pictorally represented as in Definition~\ref{def:staircase}. Using a similar argument as in the proof of Lemma~\ref{lem:box}, every $V^{k_i}$ with $k_i\geq2$ vanishes in $C_n$, so $C_n$ and $\widetilde{C}_n$ differ only in the arrows $x_i\to nU^{n-1}x_{i+1}$ coming from sources $x_i$ with $k_i=1$. Put $\varepsilon_i=1$ if $k_i=1$ and $\varepsilon_i=0$ otherwise, so that for each source $x_i$,
\[
\partial_nx_i=U^{k_{i-1}}x_{i-1}+\varepsilon_i\,nU^{n-1}x_{i+1}.
\]
Rather than cancelling over $\Q$ (which is not $U$-equivariant), we change basis over $\frac{\Q[U]}{(U^n)}$.

(1) For even $i$, set $y_{i-1}=x_{i-1}+\varepsilon_i\,nU^{n-1-k_{i-1}}x_{i+1}$ and $y_m=x_m$. This is a change of basis of the odd generators, and $\partial_nx_i=U^{k_{i-1}}y_{i-1}$. A visual representation of this change of basis is as in Figure~\ref{fig:change-of-basis}. Since the odd generators are cycles, the $\frac{\Q[U]}{(U^n)}$-linear map $\widetilde{C}_n\to C_n$ fixing $x_i$ for $i$ even and sending $x_i\mapsto y_i$ for $i$ odd is an isomorphism of chain complexes.

(2) Define recursively $x_m'=x_m$ and $x_i'=x_i-\varepsilon_i\,nU^{n-1-k_{i+1}}x_{i+2}'$ for odd $i\leq m-2$. By induction, $\partial_nx_i'=U^{k_{i-1}}x_{i-1}$: indeed $\partial_nx_{i+2}'=U^{k_{i+1}}x_{i+1}$, so the second term cancels $\varepsilon_i\,nU^{n-1}x_{i+1}$. Since the even generators are cycles, the map $\widetilde{C}_n\to C_n$ sending $x_i\mapsto x_i'$ for $i$ odd and fixing $x_i$ for $i$ even is an isomorphism of chain complexes.

We can readily compute $H_*(C_n)\cong H_*(\widetilde{C}_n)$. The complex $\widetilde{C}_n$ is simply $\frac{m-1}{2}$ pairs of towers (as in Remark~\ref{rem:towers}) with arrows given by the horizontal arrows, and one single tower with no arrows. For the left form, writing $\kappa_s=\min(k_{2s-1},n)$,
\[
H_*(C_n)\cong H_*(\widetilde{C}_n)\cong\bigoplus_{s=1}^{\frac{m-1}{2}}\left(\frac{\Q[U]}{(U^{\kappa_s})}_{<x_{2s-1}>}\oplus\frac{\Q[U]}{(U^{\kappa_s})}_{<U^{n-\kappa_s}x_{2s}>}\right)\oplus\frac{\Q[U]}{(U^n)}_{<x_m>},
\]
and similarly for the right form, with $\kappa_s=\min(k_{2s},n)$, generators $x_{2s}$ and $U^{n-\kappa_s}x_{2s+1}$, and the free summand generated by $x_1$. The isomorphisms above are $\frac{\Q[U]}{(U^n)}$-linear, so these are isomorphisms of $\Q[U]$-modules.
\end{proof}

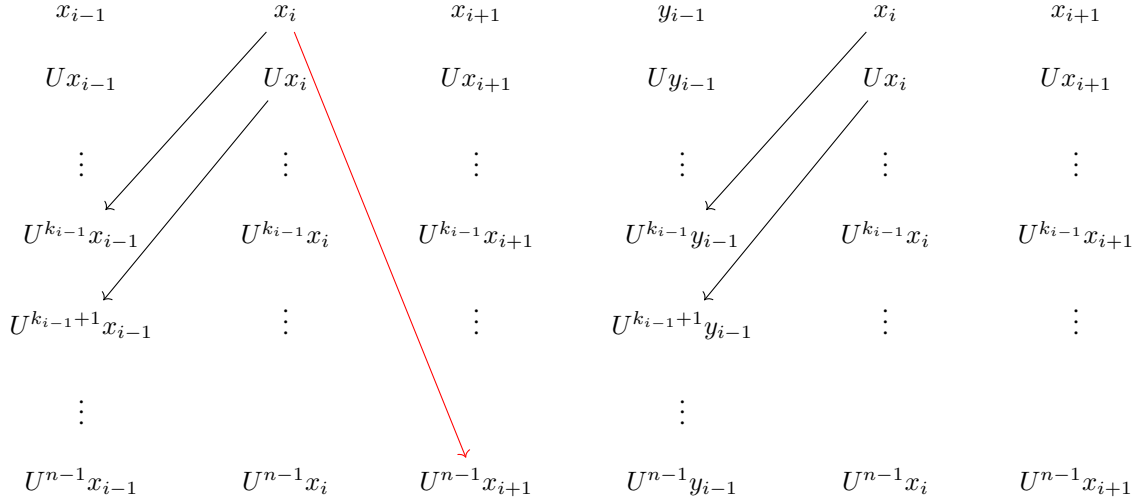
\begin{figure}[ht]
\[
\begin{tikzcd}[row sep=small, column sep=normal]
x_{i-1} & x_i \arrow[dddl] \arrow[ddddddr, red] & x_{i+1} \\
Ux_{i-1} & Ux_i \arrow[dddl] & Ux_{i+1} \\
\vdots & \vdots & \vdots \\
U^{k_{i-1}}x_{i-1} & U^{k_{i-1}}x_i & U^{k_{i-1}}x_{i+1} \\
U^{k_{i-1}+1}x_{i-1} & \vdots & \vdots \\
\vdots & & \\
U^{n-1}x_{i-1} & U^{n-1}x_i & U^{n-1}x_{i+1}
\end{tikzcd}
\qquad
\begin{tikzcd}[row sep=small, column sep=normal]
y_{i-1} & x_i \arrow[dddl] & x_{i+1} \\
Uy_{i-1} & Ux_i \arrow[dddl] & Ux_{i+1} \\
\vdots & \vdots & \vdots \\
U^{k_{i-1}}y_{i-1} & U^{k_{i-1}}x_i & U^{k_{i-1}}x_{i+1} \\
U^{k_{i-1}+1}y_{i-1} & \vdots & \vdots \\
\vdots & & \\
U^{n-1}y_{i-1} & U^{n-1}x_i & U^{n-1}x_{i+1}
\end{tikzcd}
\]
\caption{A local picture of the change of basis in the proof of Lemma~\ref{lem:staircase} (left form) for some even $i$ where $k_i=1$ and $k_{i-1}\leq n-1$. On the left: the complex $C_n$, with the arrow $x_i\to nU^{n-1}x_{i+1}$ in red. On the right: the same complex in the basis where $x_{i-1}$ is replaced by $y_{i-1}=x_{i-1}+nU^{n-1-k_{i-1}}x_{i+1}$; the red arrow has been absorbed, and the picture agrees with $\widetilde{C}_n$.}
\label{fig:change-of-basis}
\end{figure}

To conclude the main Theorem~\ref{thm:main-intro}, we make full use of the fact that, among a myriad of nice properties, all $L$-space knots and thin knots have a master complex given by direct sums of staircases and boxes as defined above:

\begin{theorem}\label{thm:main}
Let $n\geq2$, and let $K\subset S^3$ be either a thin knot, or an $L$-space knot such that every gap of $\Delta_K$ is at most $n-1$. Then, as ungraded $\Q[U]$-modules,
\[
\HFK_n(K)\cong\big(\HFK^-(K)\tensQU\QUn\big)\oplus\Tor_1^{\Q[U]}\Big(\HFK^-(K);\QUn\Big).
\]
\end{theorem}

\begin{proof}
The argument is a chain of isomorphisms given by Lemmata~\ref{lem:staircase}, \ref{lem:box}, and the universal coefficient theorem. Let $C=\CFK_{U,V}(K)$. By Theorem~\ref{thm:knot-quotient}, $\HFK_n(K)\cong H_*(C_n)$; this is an isomorphism of $\Q[U]$-modules, since Dowlin's proof exhibits the punctured complex as a mapping cone of $UV$, which after setting $V=nU^{n-1}$ is the cone of the non-zero-divisor $nU^n$ on a free $\Q[U]$-complex. By Theorems~\ref{thm:thin-structure} and~\ref{thm:lspace-structure}, $C$ is homotopy equivalent over $\Q[U,V]$ to a direct sum of one staircase complex and (for thin knots) box complexes. The constructions $C\mapsto C_n$ and $C\mapsto\widetilde{C}_n$ are given by tensoring with $\Q[U,V]$-modules (Remark~\ref{rem:tensor}), so they preserve homotopy equivalences and direct sums. For thin knots all $k_i=1\leq n-1$ and all boxes have $k=l=1\leq n-1$; for $L$-space knots the steps $k_i$ are the gaps of $\Delta_K$, all at most $n-1$ by hypothesis. Hence Lemmata~\ref{lem:staircase} and~\ref{lem:box} apply to every summand, and
\[
H_*(C_n)\cong H_*(\widetilde{C}_n)\cong H_*\Big(\CFK^-(K)\tensQU\frac{\Q[U]}{(U^n)}\Big),
\]
using $C/(V)\simeq\CFK^-(K)$ over $\Q[U]$. The universal coefficient theorem says
\[
H_*\left(\CFK^-(K)\tensQU\frac{\Q[U]}{(U^n)}\right)\cong\left[H_*\big(\CFK^-(K)\big)\tensQU\frac{\Q[U]}{(U^n)}\right]\bigoplus\Tor_1^{\Q[U]}\Big(\HFK^-(K);\frac{\Q[U]}{(U^n)}\Big).
\]
But
\[
\HFK^-(K)=H_*(\CFK^-(K))
\]
and Lemmata~\ref{lem:staircase} and~\ref{lem:box} yield
\[
\HFK_n(K)\cong H_*\Big(\CFK^-(K)\tensQU\frac{\Q[U]}{(U^n)}\Big).
\]
Piecing it all together gives the result.
\end{proof}

For $L$-space knots, the gap condition in Theorem~\ref{thm:main} is sufficient but not necessary. The following proposition computes $\HFK_n(K)$ for every $L$-space knot and every $n\geq2$. We write $R=\frac{\Q[U]}{(U^n)}$ for brevity.

Let $n\geq2$ and let $C$ be a staircase complex of the left form, with sources $x_2,x_4,\dots,x_{m-1}$. For even $i$, write $h_i=k_{i-1}$ and $v_i=k_i$ for the lengths of the horizontal and vertical arrows with source $x_i$, respectively. Call an even $i$ a \emph{defect} if $h_i\geq n$ and $v_i=1$. For a defect $i$, let $r(i)$ be the smallest even $r>i$ such that $r=m+1$ or $(h_r,v_r)\neq(n-1,1)$, and say the defect \emph{survives} if $r(i)=m+1$ or $h_{r(i)}\geq n$. Let $N_n(C)$ be the number of surviving defects. For an $L$-space knot $K$ we write $N_n(K)$ for $N_n(C)$, where $C$ is the staircase of Theorem~\ref{thm:lspace-structure}. Then we have the following:

\begin{proposition}\label{prop:sharp}
As $\Q[U]$-modules,
\[
H_*(C_n)\cong R\oplus\bigoplus_{i\text{ even}}\left(\frac{\Q[U]}{(U^{\min(h_i,n)})}\right)^2,
\]
except that $2N_n(C)$ of the summands isomorphic to $R$ are replaced by $\frac{\Q[U]}{(U^{n-1})}$. In particular,
\[
\dim_\Q H_*(C_n)=n+2\sum_{i\text{ even}}\min(h_i,n)-2N_n(C),\qquad \dim_\Q H_*(\widetilde{C}_n)=n+2\sum_{i\text{ even}}\min(h_i,n),
\]
and $H_*(C_n)\cong H_*(\widetilde{C}_n)$ if and only if $N_n(C)=0$, in which case $C_n\cong\widetilde{C}_n$ as complexes of $R$-modules.
\end{proposition}

\begin{proof}
As in the proof of Lemma~\ref{lem:staircase}, $C_n$ is a two-term complex of free $R$-modules, from the sources $x_i$ ($i$ even) to the targets $x_j$ ($j$ odd), with $\partial_nx_i=U^{h_i}x_{i-1}+\varepsilon_i\,nU^{n-1}x_{i+1}$, where $U^{h_i}=0$ if $h_i\geq n$. We record it as a matrix $D$ over $R$, with a column for each source and a row for each target; $\widetilde{C}_n$ has the matrix $\widetilde{D}$ with only the entries $U^{h_i}$ in positions $(i-1,i)$.

We process $i=2,4,\dots,m-1$ in order, by invertible $R$-linear changes of basis of the sources (column operations) and of the targets (row operations). We maintain the following invariant: when column $i$ is reached, row $i-1$ contains, besides its entry in column $i$, at most one other nonzero entry, of the form $cU^{n-1}$ with $c\in\Q^\times$, in an earlier column $j$ which has no other nonzero entry. We call it the \emph{pending defect}.
\begin{itemize}
\item If $h_i\leq n-1$ and there is a pending defect at $(i-1,j)$, replace $x_j$ by $x_j-cU^{n-1-h_i}x_i$. This kills the entry $(i-1,j)$ and writes $-c\,\varepsilon_i\,nU^{2n-2-h_i}$ at $(i+1,j)$, which is nonzero exactly when $h_i=n-1$ and $\varepsilon_i=1$; in that case it is the pending defect for row $i+1$. Then, as in Lemma~\ref{lem:staircase}, replace $x_{i-1}$ by $x_{i-1}+\varepsilon_i\,nU^{n-1-h_i}x_{i+1}$; since row $i-1$ now has a single entry, this kills $(i+1,i)$ and changes nothing else.
\item If $h_i\geq n$, column $i$ consists of the single entry $\varepsilon_i\,nU^{n-1}$ at $(i+1,i)$. If $i$ is a defect this is the pending defect for row $i+1$. A pending defect in row $i-1$, if any, is now final, since row $i-1$ has no other nonzero entry.
\end{itemize}
A new pending defect lands in row $i+1$ only if either $h_i\geq n$ or $(h_i,v_i)=(n-1,1)$, and these are exclusive, so each row carries at most one; this verifies the invariant. A defect $i$ is thus carried along the rows $i+1,i+3,\dots$ while $(h_r,v_r)=(n-1,1)$, and at the first even $r>i$ with $(h_r,v_r)\neq(n-1,1)$ it is absorbed if $h_r\leq n-1$ and becomes final if $h_r\geq n$; if no such $r\leq m-1$ exists, it becomes final in row $m$. So the final defects are exactly the surviving ones. (In the language of Lemma~\ref{lem:cancellation}, the column operation is the zigzag term created when the arrow $x_i\to U^{h_i}x_{i-1}$ is cancelled; this is the term missed by a naive cancellation argument.)

All operations are invertible and $R$-linear, and the final matrix is $\widetilde{D}$ together with one entry $cU^{n-1}$ for each surviving defect, placed in a column and a row which are zero in $\widetilde{D}$. Its homology is a direct sum of: $\big(\frac{\Q[U]}{(U^{h_i})}\big)^2$ for each $i$ with $h_i\leq n-1$, from $R\xrightarrow{U^{h_i}}R$; $\big(\frac{\Q[U]}{(U^{n-1})}\big)^2$ for each surviving defect, as in Lemma~\ref{lem:counterexample}; and one copy of $R$ for each remaining zero row or column. Comparing with $\widetilde{D}$, each surviving defect replaces two copies of $R$ by two copies of $\frac{\Q[U]}{(U^{n-1})}$, which gives the statement.
\end{proof}

\begin{remark}\label{rem:mirror}
The staircase of the right form with steps $k_1,\dots,k_{m-1}$ is the dual of the staircase of the left form with steps $k_{m-1},\dots,k_1$, and the master complex of the mirror of a knot is the dual of its master complex. Since $C_n$ is a two-term complex over the principal ideal ring $R$, its homology is determined by the Smith normal form of its matrix, which is unchanged by transposition. Hence Proposition~\ref{prop:sharp} also computes $H_*(C_n)$ for staircases of the right form, and in particular $\HFK_n$ of mirrors of $L$-space knots.
\end{remark}

We illustrate Proposition~\ref{prop:sharp} with two examples. The first is a staircase complex, not the master complex of a knot, chosen to show each of the three possibilities for a defect: it may survive, be carried to a later row, or be absorbed.
 
\begin{example}\label{ex:defects}
Let $C$ be the staircase complex of the left form with $m=9$ and steps
\[
(k_1,\dots,k_8)=(3,1,1,1,1,1,2,1),
\]
so that the sources $x_2,x_4,x_6,x_8$ have
\[
(h_2,v_2)=(3,1),\qquad(h_4,v_4)=(1,1),\qquad(h_6,v_6)=(1,1),\qquad(h_8,v_8)=(2,1).
\]
 
\emph{The case $n=2$.} Here $R=\frac{\Q[U]}{(U^2)}$, and in $C_2$,
\[
\partial_2x_2=2Ux_3,\quad\partial_2x_4=Ux_3+2Ux_5,\quad\partial_2x_6=Ux_5+2Ux_7,\quad\partial_2x_8=2Ux_9,
\]
since $U^3=U^2=0$. With columns indexed by the sources $x_2,x_4,x_6,x_8$ and rows by the targets $x_1,x_3,x_5,x_7,x_9$, the matrices $D$ and $\widetilde{D}$ of the proof of Proposition~\ref{prop:sharp} are
\[
D=\begin{bmatrix}0&0&0&0\\2U&U&0&0\\0&2U&U&0\\0&0&2U&0\\0&0&0&2U\end{bmatrix},\qquad
\widetilde{D}=\begin{bmatrix}0&0&0&0\\0&U&0&0\\0&0&U&0\\0&0&0&0\\0&0&0&0\end{bmatrix}.
\]
Both $x_2$ and $x_8$ are defects. Following the proof of Proposition~\ref{prop:sharp}, the entry $2Ux_3$ of $\partial_2x_2$ is the pending defect in row $x_3$. At $x_4$ we have $h_4=1\leq n-1$, so we replace $x_2$ by $x_2-2x_4$; then
\[
\partial_2(x_2-2x_4)=-4Ux_5,
\]
which is nonzero because $(h_4,v_4)=(n-1,1)$, so the defect is carried to row $x_5$. Replacing $x_3$ by $y_3=x_3+2x_5$ gives $\partial_2x_4=Uy_3$. In the same way, at $x_6$ we replace $x_2-2x_4$ by $x_2'=x_2-2x_4+4x_6$, so that $\partial_2x_2'=8Ux_7$, and $y_5=x_5+2x_7$ gives $\partial_2x_6=Uy_5$. At $x_8$ we have $h_8=2\geq n$, so the entry $8Ux_7$ cannot be cancelled: the defect of $x_2$ survives, with $r(2)=8$. The defect of $x_8$ survives as well, since $r(8)=m+1$. Thus $N_2(C)=2$. In the bases $\{x_2',x_4,x_6,x_8\}$ and $\{x_1,y_3,y_5,x_7,x_9\}$, the matrix of $\partial_2$ is
\[
\begin{bmatrix}0&0&0&0\\0&U&0&0\\0&0&U&0\\8U&0&0&0\\0&0&0&2U\end{bmatrix},
\]
which is $\widetilde{D}$ together with the two entries $8U$ and $2U$ coming from the surviving defects. Each nonzero entry contributes $\Q\oplus\Q$ to the homology, and the zero row $x_1$ contributes $R$. Hence
\[
H_*(C_2)\cong\frac{\Q[U]}{(U^2)}\oplus\Q^8,\qquad H_*(\widetilde{C}_2)\cong\Big(\frac{\Q[U]}{(U^2)}\Big)^5\oplus\Q^4,
\]
of dimensions $10$ and $14$, in agreement with $n+2\sum\min(h_i,n)-2N_2(C)=2+2(2+1+1+2)-4=10$.
 
\emph{The case $n=3$.} Here $R=\frac{\Q[U]}{(U^3)}$ and $\partial_3x_2=3U^2x_3$, so $x_2$ is still a defect, but now $x_8$ is not, since $h_8=2<n$. At $x_4$ we replace $x_2$ by $x_2-3Ux_4$, and
\[
\partial_3(x_2-3Ux_4)=3U^2x_3-3U(Ux_3+3U^2x_5)=-9U^3x_5=0.
\]
The defect is absorbed, because $(h_4,v_4)=(1,1)\neq(n-1,1)$; in the notation of Proposition~\ref{prop:sharp}, $r(2)=4$ and $h_4\leq n-1$. The remaining vertical arrows are absorbed as in Lemma~\ref{lem:staircase}, so $N_3(C)=0$, $C_3\cong\widetilde{C}_3$, and
\[
H_*(C_3)\cong H_*(\widetilde{C}_3)\cong\Big(\frac{\Q[U]}{(U^3)}\Big)^3\oplus\Big(\frac{\Q[U]}{(U^2)}\Big)^2\oplus\Q^4,
\]
which is dimension $17$.
\end{example}
 
\begin{example}\label{ex:T57}
The torus knot $T_{5,7}$ has genus $12$, and the gaps of $\Delta_{T_{5,7}}$ are
\[
(k_1,\dots,k_{16})=(1,4,1,1,1,2,1,1,1,1,2,1,1,1,4,1).
\]
So its complex is a staircase like that of Theorem~\ref{thm:lspace-structure}, which has $m=17$, and its sources $x_2,x_4,\dots,x_{16}$ have
\[
(h_i)=(1,1,1,1,1,2,1,4),\qquad(v_i)=(4,1,2,1,1,1,1,1).
\]
Take $n=2$. The defects are $x_{12}$, with $(h_{12},v_{12})=(2,1)$, and $x_{16}$, with $(h_{16},v_{16})=(4,1)$. The defect of $x_{12}$ is carried through $x_{14}$, since $(h_{14},v_{14})=(1,1)=(n-1,1)$, and survives at $x_{16}$, since $h_{16}\geq n$; the defect of $x_{16}$ survives since $r(16)=m+1$. So $N_2(T_{5,7})=2$, and by Proposition~\ref{prop:sharp},
\[
\HFK_2(T_{5,7})\cong\frac{\Q[U]}{(U^2)}\oplus\Q^{16},
\]
which has dimension $18$, while the right-hand side of Theorem~\ref{thm:main} has dimension $22$. For $n=3$ only the defect of $x_{16}$ remains, so $N_3(T_{5,7})=1$ and the dimensions are $23$ and $25$; for $n=4$ again $N_4(T_{5,7})=1$, with dimensions $26$ and $28$. For $n\geq5$ there are no defects, in accordance with Corollary~\ref{cor:torus}.
\end{example}

\begin{remark}\label{rem:thin-formula}
For a thin knot $K$ with $b$ boxes as in Theorem~\ref{thm:thin-structure}, $\HFK^-(K)\cong\Q[U]\oplus\Q^{|\tau(K)|+2b}$, and $\rk\widehat{\HFK}(K)=2|\tau(K)|+1+4b=\det(K)$, since $\widehat{\HFK}(K)$ is supported in a single $\delta$-grading. By Theorem~\ref{thm:main}, for $n\geq2$,
\[
\HFK_n(K)\cong\frac{\Q[U]}{(U^n)}\oplus\Q^{\frac{\det(K)-1}{2}}\oplus\Q^{\frac{\det(K)-1}{2}}=\frac{\Q[U]}{(U^n)}\oplus\Q^{\det(K)-1},
\]
which proves the second part of Corollary~\ref{cor:determined}. For $2$-bridge knots and $n=2$, this agrees with Dowlin's computation $\dim_\Q\HFK_2(K)=\det(K)+1$ \cite[Lemma 5.4]{Dow18}.
\end{remark}

\begin{proof}[Proof of Corollary~\ref{cor:torus}]
For $T_{p,q}$, let $S\subset\mathbb{Z}_{\geq0}$ be the semigroup generated by $p$ and $q$. The identity $\frac{1-t^{pq}}{(1-t^p)(1-t^q)}=\sum_{s\in S}t^s$ gives $\Delta_{T_{p,q}}(t)=(1-t)\sum_{s\in S}t^s$, normalized to have lowest degree $0$. So the exponents of $\Delta_{T_{p,q}}$ are the $s\in S$ with $s-1\notin S$ and the $s+1\notin S$ with $s\in S$, and the gaps are the lengths of maximal sequences of consecutive elements of $S$ and of consecutive non-elements of $S$. A sequence of non-elements has length at most $p-1$, since $s\in S$ implies $s+p\in S$, and the sequence $\{1,\dots,p-1\}$ attains this. Since $\Delta$ is symmetric, its sequence of gaps is a palindrome of even length, so the sequences of elements are the sequences of non-elements in reverse order. Hence the largest gap is length $p-1$.

If $n\geq p$, every gap is at most length $n-1$, and Theorem~\ref{thm:main} applies. If $2\leq n\leq p-1$, the lowest terms of $\Delta$ are $1-t+t^p$, so the first two gaps are $1$ and $p-1$, and by symmetry so are the last two. In its staircase of the shape in Theorem~\ref{thm:lspace-structure}, the last source $x_{m-1}$ therefore has $(h_{m-1},v_{m-1})=(p-1,1)$ with $p-1\geq n$; it is a defect with $r(m-1)=m+1$, so it survives and $N_n(T_{p,q})\geq1$. By Proposition~\ref{prop:sharp}, $\dim_\Q\HFK_n(T_{p,q})$ is smaller than the dimension of the right-hand side of Theorem~\ref{thm:main}. The case of $T_{3,4}$ is Example~\ref{ex:T34}.
\end{proof}

The same argument shows more generally that, since the top gap of the Alexander polynomial of a nontrivial $L$-space knot is $1$ \cite{HW18}, the isomorphism of Theorem~\ref{thm:main} fails for every nontrivial $L$-space knot whose second gap is at least $n$.

\bibliographystyle{amsalpha}
\bibliography{bib}

\end{document}